\documentclass{iasr}
\renewcommand\thisnumber{0}
\renewcommand\thisyear {2010}
\renewcommand\thismonth{February}
\renewcommand\thisvolume{00}
\usepackage[dvips]{graphicx}
\usepackage{amsthm}
\newtheorem{prob}{Problem}[section]

\begin{document}

%\title{On an existence result for a non-isothermal,
%non-Newtonian strongly coupled problem with heat convection term and Tresca's
%law}
\title{Existence result for a  strongly coupled problem with heat
convection term and Tresca's
law}

\author[Ahmed Bensedik  A. and Boukrouche M.,]{Bensedik
Ahmed\affil{1, 2}, %\comma\corrauth,
       Boukrouche Mahdi\affil{1}}
 \address{\affilnum{1} \ Lyon University, F-42023 UJM, LaMUSE EA-3989, 23 Paul
Michelon 42023 Saint-Etienne, France.
 Email: {\tt Mahdi.Boukrouche@univ-st-etienne.fr}\\
 \affilnum{2}\ Faculty of Sciences,  Department of Mathematics, University of
Tlemcen, Algeria. E-mail: \tt  a\_bensedik@mail.univ-tlemcen.dz}
 \corraddr{Mahdi Boukrouche, Lyon University, F-42023 UJM,
                         LaMUSE EA-3989, 23 rue Paul Michelon 42023
Saint-Etienne,
                         France. Email: {\tt
Mahdi.Boukrouche@univ-st-etienne.fr}\\
         }
\received{3 october~ 2009}
%\accepted{2 november~ 2009}

\newtheorem{thm}{Theorem}[section]
 \newtheorem{cor}[thm]{Corollary}
 \newtheorem{lem}[thm]{Lemma}
 \newtheorem{prop}[thm]{Proposition}
 \newtheorem{defn}[thm]{Definition}
 \newtheorem{rem}[thm]{Remark}
\newtheorem{Example}{Example}[section]
\newcommand{\R}{\mathbb{R}}

%%%%% Begin Abstract %%%%%%%%%%%
\begin{abstract}
We study a problem describing the motion of an incompressible,
non-isothermal and non-Newtonian fluid, taking into account the heat convection
term.
The novelty here is that  fluid viscosity  depends on  the temperature,
the velocity of the fluid, and also of  the deformation tensor,
but not explicitly. The boundary conditions take into account the slip
phenomenon
on a part of the boundary of the domain. By using the notion of
pseudo-monotone operators and fixed point Theorem we prove an existence
result of its weak solution.
\end{abstract}

\keywords{Heat convection;
 Non-Newtonian fluid; Non-isothermal fluid; Tresca fluid-solid conditions;
Pseudo-monotone operators; Schauder point fixed Theorem.}
\msc{76A05, 76D50, 35Q35.}

%%%% maketitle %%%%%
\maketitle
%
%%%% Start %%%%%%
\renewcommand{\theequation}{1.\arabic{equation}}
\setcounter{equation}{0}
\section{Introduction}\label{sec1}
Let $\omega$ be fixed bounded domain of $\R^2$, with Lipschitz continuous
boundary. We suppose that $\omega$ is the bottom of the fluid domain $\Omega$,
the upper surface $\Gamma_{1}$ is defined by the equation $x_{3}=h(x')$ where
$x'=(x_{1},x_{2})$ and $h$ is a positive smooth and bounded function. Then
\begin{equation*}
\Omega=\left\{\left(x',x_{3}\right )\in \R^3; \quad x'\in \omega \quad and \quad
0<x_{3}<h\left(x'\right)\right\}.
\end{equation*}
The boundary $\partial\Omega$ is composed of three parts;
$\partial\Omega=\overline{\omega}\cup\overline{\Gamma}_{1}\cup\overline{\Gamma}_
{L}$,
 where $\Gamma_{L}$ is the lateral boundary.
We consider a stationary problem, in the bounded domain $\Omega$, describing
the motion of an incompressible non-isothermal and non-Newtonian fluid.
This problem is deduced (see \cite{MB2009}) from the three conservation laws of,
mass,
momentum and energy, (see for example \cite{dlions, majda1984}),
where the density is assumed to be constant and equal to 1, so the mass
conservation
law becomes the incompressibility condition of the fluid
\begin{eqnarray}\label{3}
div\left(v\right)=0 \quad in \quad \Omega,
\end{eqnarray}
where $v$ is the velocity of the fluid.  Many fluid flows (molten polymers in
solution, oils,
sludge ...) do not verify Newton's law
$\sigma(v)=-\pi I+2\mu D(v)$, with $\mu =const.$,
 but a more complex in which the viscosity $\mu$ varies with the strain tensor
$D(v)$, the temperature $\theta$, or also the second invariant
 $D_{11}={1\over2}D(v):D(v)$.
We consider here that the heat conduction phenomenon is described by Fourier's
law, relating the heat flux to the temperature $\theta$, so the energy law leads
to
\begin{eqnarray}\label{2}
v.\nabla \theta=2\mu\left(\theta,v,
\left|D\left(v\right)\right|\right)D\left(v\right):D\left(v\right)+div(K\nabla
\theta)+r\left(\theta\right) \quad in \quad\Omega,
\end{eqnarray}
where $K$ is a positive function defined on $\Omega$ and $r$ is real function.

The motion of the fluid is assumed to be slow, then the momentum law leads to
\begin{eqnarray}\label{1}
-2 div\left(\mu\left(\theta,v,\left|D(v)\right|\right) D(v)\right)+\nabla \pi=f
\quad in \quad\Omega,
\end{eqnarray}
where $\pi$ is the pressure of the fluid, $f$ is a given vector and will be
specified later.\newline

Compared to the works \cite{gb-mb1}-\cite{sf1},  and  to the  earlier works,  to
our knowledge, the novelty in this study is firstly
that we take into account the effects
of the heat convection expressed by the presence of the left term in  (\ref{2}),
secondly we consider the viscosity $\mu$ of
the fluid as a function depending on its temperature,
its velocity and its strain tensor. This general choice of the viscosity allows
us to include the cases,
of power law  \cite{gwia, ge}, Carreau law \cite{tapiero2, tapiero} or Bingham
law \cite{dlions}.

This choice also allows to choose the appropriate viscosity that meets
industrial
 applications such that, for example, the manufacture of flak vest, containing a
fluid that has the ability to focus  on the impact of the projectile upon
contact
 with the flak vest.

See an other situation in \cite{malek1} where the paper concerns longtime and
large-data existence results for a generalization of the
Navier-Stokes fluid whose viscosity depends on the shear rate and the pressure
in the form $\nu= \nu(p , |D(v)|^{2})$. We can see  also   \cite{malek2}
 in which the authors considered  flows of fluids in  unbounded domains
with  viscosity having the same form as in   \cite{malek1}.\newline\

In order to close the system, it remains to describe the boundary conditions for
the temperature $\theta$ and the velocity $v$.
We consider the  following  Neumann  boundary condition
\begin{eqnarray}\label{11}
K\frac{\partial \theta}{\partial n}=\theta_{\omega},\quad on \quad \omega,
\end{eqnarray}
where $n=(n_{1},n_{2},n_{3})$ is the outward unit normal to $\partial\Omega$,
and $\theta_{\omega}$
 is a given fixed temperature flux on $\omega$,
and the following Dirichlet condition
\begin{eqnarray}\label{12}
\theta=0,\quad on \quad \Gamma_{1}\cup\Gamma_{L}.
\end{eqnarray}
For the boundary conditions for the velocity $v$,  let $g=(g_{1},g_{2},g_{3})$
be a function such that
\begin{eqnarray*}
\int_{\partial\Omega}g.n ds=0,\quad g_{3}=0 \quad on \quad \Gamma_{L},\quad g=0
\quad on \quad \Gamma_{1}, \quad g.n=0,\quad on\quad \omega,
\end{eqnarray*}
 the velocity on $\Gamma_{L}$ is known and parallel to the plane $(x_{1},x_{2})$
so,
\begin{eqnarray}\label{6}
v=g\quad on \quad \Gamma_{L},
\end{eqnarray}
the upper surface $\Gamma_{1}$ being assumed to be fixed so
\begin{eqnarray}\label{5}
v=0\quad on \quad \Gamma_{1},
\end{eqnarray}
we assume that there is no flux through $\omega$, so the normal component of
the velocity vanishes,
\begin{eqnarray}\label{7}
v.n=0\quad on \quad \omega,
\end{eqnarray}
but the tangential component $v_{t}$ of the velocity is unknown and satisfies
the Tresca law \cite{gb-mb1}, \cite{dlions} Chap.3,

\begin{eqnarray}\label{8}
\left\{
\begin{array}{ll}
\left|\sigma_{t}\right|=k \Rightarrow \exists\lambda \geq 0;
v_{t}=s-\lambda\sigma_{t},\\
\left|\sigma_{t}\right|<k \Rightarrow v_{t}=s.
\end{array}
\right.
\end{eqnarray}
Where $k$ is the upper limit for the stress, $s$ is the velocity of the surface
$\omega$ and $\sigma_{t}$ is the tangential component of $\sigma n$, where
$\sigma=(\sigma_{ij})_{1\leq i,j \leq 3}$ is the stress tensor defined by
\begin{eqnarray}\label{9}
\sigma_{ij}=2\mu\left(\theta,v,\left|D\left(v\right)\right|\right)d_{ij}
\left(v\right)-\pi\delta_{ij},
\end{eqnarray}
with $\delta_{ij}$ is the Kronecker symbol and $D(v)$ is the deformation tensor
given by
\begin{eqnarray}\label{10}
D(v)=(d_{ij}(v))_{1\leq i,j\leq 3}, \quad d_{ij}(v)=\frac{1}{2}(\partial_{j}
v_{i}+\partial_{i} v_{j}).
\end{eqnarray}
Here and below, we write $\partial_{i}$ to design $\frac{\partial }{\partial
x_{i}}$
and use the convention of implicit summation over repeated indices.

The term $\left|D(v)\right|$
denotes the euclidian norm of $D(v)$, that is
$\left|D(v)\right|^2=d_{ij}(v)d_{ij}(v)$,
 induced by the product $D(u):D(v)=d_{ij}(u)d_{ij}(v)$. \newline

The plan of this paper is as follows.
In Section \ref{secWF} we establish the variational formulation of the
considered problem
(\ref{3})-(\ref{8}). Note that the terms $\mu\left(\theta,v,
\left|D\left(v\right)\right|\right)D\left(v\right):D\left(v\right)$ and
 $v.\nabla \theta$ in (\ref{2}) and
$div\left(\mu\left(\theta,v,\left|D(v)\right|\right) D(v)\right)$
 in (\ref{1}) lead us to consider particular sets of the Sobolev spaces
$\left(W^{1,p}\left(\Omega\right)\right)^3$ and
$\left(W^{1,q}\left(\Omega\right)\right)^3$
where $p>3$ and $q$ its exponent conjugate, for details see the proof of
Proposition \ref{prob2.1}.
In the same section we give some lemmas needed for the next one to prove the
boundedness
and the coerciveness of the operator $A$ defined by
\begin{eqnarray*}
\left\langle
A(v),\varphi\right\rangle=\int_{\Omega}2\mu\left(\theta,v,|D(v)|\right)d_{ij}
(v)\partial_{j}\varphi_{i} dx.
\end{eqnarray*}
 In fact its coerciveness also follows from the assertion
\begin{eqnarray*}
\lim_{\left\|v\right\|_{1,2}\rightarrow
+\infty}\frac{\left\|v\right\|_{1,2}}{\left\|v\right\|_{1,p}}\neq0,
\end{eqnarray*}
which we prove in the present paper, for notations, see the next section.
Also note the fact that the function $\mu$ does not depend explicitly on its
arguments leads us to assume it monotone in $\left|D(v)\right|$.
This enable us to establish the monotonicity of the operator $A$.

In Section \ref{sec3} we study the existence results of the problem
\ref{prob2.1}, in the following
three subsections.

In Subsection \ref{sec3.1}, for given temperature $\theta\in W^{1,
q}_{\Gamma_{1}\cup\Gamma_{L}}(\Omega)$
and $f\in (W^{1, p}(\Omega))^{3}$, $0 \leq k \in L^{p}(\omega)$,
we prove in Theorem \ref{th3.1}, using Korn inequalities and classical results
of monotone operators,
 that there exists at least one
$v_{\theta}\in W^{1, p}_{div}(\Omega)$ solution of the intermediate problem
\ref{prob3.1}.
 We prove also in Lemma \ref{lem3.2} that $v_{\theta}$ remains bounded in
$W^{1, p}_{div}(\Omega)$, independently of the temperature $\theta$.
This is needed in Theorem \ref{th3.2} where we derive the existence of the
pressure $\pi\in L^{p}_{0}(\Omega)$ satisfying the varitional inequality
(\ref{14}).\newline

In Subsection \ref{sec3.2}, We consider a second intermediate problem
that, for given velocity $v\in V^{p}_{div}$ find the temperature solution of
Problem \ref{Sprob}.
 We remark here that from the weak formulation, the temperature $\theta$ must
be obtained in the space $W^{1,q}\left(\Omega\right)$. But this weak formulation
led us, after linearization of the corresponding equation, to a bilinear form
$B\left(\theta,\psi\right)$ defined on the space
$W^{1,q}\left(\Omega\right)\times W^{1,p}\left(\Omega\right)$.
And in order to apply Lax-Milgram Theorem, we have considered
the Hilbert space $H^1\left(\Omega\right)\times H^1\left(\Omega\right)$,
and we establish in Theorem \ref{thm3.3} the existence and uniqueness of
$\theta$ in $H^1\left(\Omega\right)$,
solution of the linearized problem \ref{pro3.2},
hence $\theta\in W^{1,q}\left(\Omega\right)$, because $1<q<2$, see other reasons
for this
choice of the space $H^1\left(\Omega\right)$ in the proof of Theorem
\ref{thm3.3}.

Using Schauder Fixed Point Theorem, we establish first in Theorem \ref{thm3.4},
 the existence of $\theta\in
H^{1}_{\Gamma_{1}\cup\Gamma_{L}}\left(\Omega\right)$
solution of the second intermediate problem (\ref{Sprob}), then
 we prove its uniqueness  in Theorem \ref{thm3.5} by a monotony method.

In Subsection \ref{sec3.3}, We recapitulate the necessary assumptions used to
prove the existence of at least one solution
of the variational global problem \ref{prob2.1}, and establish it in
Theorem \ref{thm3.6}.

\renewcommand{\theequation}{2.\arabic{equation}}
\setcounter{equation}{0}
\section{Weak formulation}\label{secWF}
Let $p$ be a real number such that $p>1$, and $q$ its conjugate exponent, that
is $\frac{1}{p}+\frac{1}{q}=1$. We know from Lemma 2.2 in \cite{gr} that for
$g\in \left(W^{1-\frac{1}{p},p}\left(\partial\Omega\right)\right)^3$, there
exists a function $G$ in
$\left(W^{1,p}\left(\Omega\right)\right)^3$ such that,
\begin{eqnarray*}
div(G)=0 \quad in\quad \Omega,\quad G=g \quad on \quad
\Gamma_{L}\cup\Gamma_{1},\quad G.n=0 \quad on \quad \omega.
\end{eqnarray*}
To establish a weak formulation of the problem, we introduce the following sets,
\begin{eqnarray*}
V^p=\left\{\varphi\in \left(W^{1,p}\left(\Omega\right)\right)^3; \quad \varphi=G
\quad on \quad \Gamma_{L},\quad \varphi=0 \quad on \quad \Gamma_{1} \quad and
\quad  \varphi.n=0 \quad on \quad \omega\right\},
\end{eqnarray*}
\begin{eqnarray*}
V^p_{div}=\left\{\varphi\in V^p; \quad div\left(\varphi\right)=0 \quad in \quad
\Omega\right\},
\end{eqnarray*}
$V^p$ and $V^p_{div}$ are convex closed subsets of
$\left(W^{1,p}\left(\Omega\right)\right)^3$.
We also define the spaces,
\begin{eqnarray*}
V^{p}_{0}=\left\{\varphi\in \left(W^{1,p}\left(\Omega\right)\right)^3; \quad
\varphi=0,  \quad on \quad \Gamma_{L}\cup\Gamma_{1}\quad and \quad \varphi.n=0
\quad on \quad \omega\right\}
\end{eqnarray*}
\begin{eqnarray*}
V^p_{0, div}=\left\{\varphi\in V^{p}_{0};\quad div\left(\varphi\right)=0\quad in
\quad \Omega\right\},
\end{eqnarray*}
\begin{eqnarray*}
L^p_{0}\left(\Omega\right)=\left\{u\in L^p\left(\Omega\right); \int_{\Omega}\
u(x)dx=0\right\},
\end{eqnarray*}
and we denote,
\begin{eqnarray*}
V^p_{\Gamma}=\left\{\varphi\in \left(W^{1,p}\left(\Omega\right)\right)^3;\quad
\varphi=0 \quad on \quad \Gamma\right\},
\end{eqnarray*}
\begin{eqnarray*}
W^{1,p}_{\Gamma}\left(\Omega\right)=\left\{\varphi\in
W^{1,p}\left(\Omega\right);\quad \varphi=0 \quad on \quad \Gamma\right\},
\end{eqnarray*}
where $\Gamma$ is a subset of $\partial\Omega$ with
$\left|\Gamma\right|:=meas(\Gamma)\neq0$.\newline
Remark that $V^p_{\Gamma}=\left(W^{1,p}_{\Gamma}\left(\Omega\right)\right)^3$
and $V^p, V^p_{0}\subset V^p_{\Gamma_1}$.
We denote, as usual the norm of the Lebesgue space $L^p\left(\Omega\right)$ by
$\left\|v\right\|_p=\left(\int_\Omega\left|v\right|^p dx \right)^\frac{1}{p}$,
and the norms of the Banach spaces $V^p_{\Gamma}$ and
$W^{1,p}_{\Gamma}\left(\Omega\right)$
are designed equally by
%\begin{eqnarray*}
$\left\|v\right\|_{1,p}=\left(\int_\Omega\left|\nabla v\right|^p dx
\right)^\frac{1}{p}$.
%\end{eqnarray*}

  \bigskip

In the following proposition, we deduce rigorously the variational formulation
of Problem (\ref{3})-(\ref{8}).
Also note that we only give the necessary assumptions on the data of each part
of the paper. We summarize, the necessary hypotheses needed for the final result
in Subsection \ref{sec3.3}.

\begin{prop}
Let $p>3$ and $q$ its conjugate exponent. For $f\in(W^{1,p}(\Omega))^3$, $0\leq
k\in L^p(\omega)$,
$\mu\in L^{\infty}\left(\R\times\R^3\times \R_+\right)$, $K\in L^\infty(\Omega)$
and $r\in L^\infty\left(\R\right)$, the weak formulation of the strong problem
{\rm(\ref{3})-(\ref{8})} leads to the following variational problem,
\end{prop}

\begin{prob}\label{prob2.1}
Find $v\in V^p_{div}$, $\pi\in L^p_0\left(\Omega\right)$ and $\theta\in
W^{1,q}_{\Gamma_{1}\cup\Gamma_{L}}\left(\Omega\right)$ such that
\begin{eqnarray}\label{14}
a\left(\theta, v, \varphi-v\right)-\left(\pi,
div(\varphi)\right)+j\left(\varphi\right)-j\left(v\right)\geq\left(f,
\varphi-v\right)\quad \forall \varphi\in V^q,
\end{eqnarray}
\begin{eqnarray}\label{15}
B\left(\theta,\psi\right)=L\left(\theta,\psi\right)\quad \forall\psi \in
W^{1,p}_{\Gamma_{1}\cup\Gamma_{L}}\left(\Omega\right),
\end{eqnarray}
\end{prob}
where
\begin{eqnarray*}
a\left(\theta, v,
\varphi-v\right)=\int_{\Omega}2\mu(\theta,v,|D(v)|)d_{ij}(v)\partial_{j}
\left(\varphi_{i}-v_{i}\right) dx,
\end{eqnarray*}
\begin{eqnarray}\label{L}
L\left(\theta,\psi\right)=2\int_{\Omega}\mu\left(\theta,v,
\left|D\left(v\right)\right|\right)\left|D\left(v\right)\right|^2 \psi
dx+\int_{\Omega}r\left(\theta\right)\psi dx+\int_{\omega}\theta_{\omega}\psi
dx',
\end{eqnarray}
and
\begin{eqnarray}\label{B}
B\left(\theta,\psi\right)=\int_{\Omega}K(x) \nabla\theta.\nabla\psi
dx+\int_{\Omega}\psi v_{i}\partial_{i}\theta dx,\quad
j\left(\varphi\right)=\int_{\omega}k\left|\varphi-s\right|dx'.
\end{eqnarray}

\begin{proof}
Observe first that since $p>3$ and $\Omega$ is bounded, if $v=(v_1, v_2, v_3)\in
V^p$
then $v_i\in L^\infty\left(\Omega\right)$.
Hence for $\psi \in W^{1,p}_{\Gamma_{1}\cup\Gamma_{L}}\left(\Omega\right)$
and $\theta\in W^{1,q}_{\Gamma_{1}\cup\Gamma_{L}}\left(\Omega\right)$
the second integral in $B\left(\theta,\psi\right)$ makes sense.
For the first integral in $L\left(\theta,\psi\right)$
we have $\left|D\left(v\right)\right|^2\in L^{\frac{p}{2}}\left(\Omega\right)$
then
$\mu\left(\theta,v,
\left|D\left(v\right)\right|\right)\left|D\left(v\right)\right|^2 \psi \in
L^1(\Omega)$
until $\mu$ is bounded.
Also $a\left(\theta, v, \varphi-v\right)$ is well defined since
 $d_{ij}(v)\in L^p\left(\Omega\right)$ and
$\partial_{j}\left(\varphi_{i}-v_{i}\right)\in L^q\left(\Omega\right) -
L^p\left(\Omega\right)\subset L^q\left(\Omega\right)$, because $1<q<p$.\newline
To obtain the variational inequality (\ref{14}), we have from (\ref{1}) and
(\ref{9})
\begin{eqnarray}\label{16}
-\partial_{j} \sigma_{ij}=f_{i},\quad i=1, 2, 3.
\end{eqnarray}
Multiplying (\ref{16}) by $\varphi_{i}-v_{i}$ where
$\varphi=\left(\varphi_{1},\varphi_{2},\varphi_{3}\right)\in V^q$ and
integrating over $\Omega$, we obtain
\begin{eqnarray}\label{17}
\int_{\Omega}\sigma_{ij}\partial_{j}\left(\varphi_{i}-v_{i}\right)
dx-\int_{\partial\Omega}\sigma_{ij}\left(\varphi_{i}-v_{i}\right)n_{j}ds=\int_{
\Omega}f_{i}\left(\varphi_{i}-v_{i}\right)dx.
\end{eqnarray}
Since $\varphi_{i}-v_{i}=0$ on $\Gamma_{1}\cup \Gamma_{L}$,
\begin{eqnarray*}
\int_{\partial\Omega}\sigma_{ij}\left(\varphi_{i}-v_{i}\right)n_{j}ds=\int_{
\omega}\sigma_{ij}\left(\varphi_{i}-v_{i}\right)n_{j}dx'.
\end{eqnarray*}
Remark that $\sigma_{ij} n_{j}$ is the i-th component of the vector $\sigma n$,
which can be written on the form $\sigma n=\sigma_{t}+\sigma_{n}n$,
with $\sigma_{t}=(\sigma_{t_{1}},\sigma_{t_{2}},\sigma_{t_{3}})$
and $\sigma_{n}=\sigma n.n$, from which we have
%\begin{eqnarray}\label{18}
$\sigma_{ij} n_{j}=\sigma_{t_{i}}+\sigma_{n}n_{i}$.
%\end{eqnarray}
Using this equality, we get
\begin{eqnarray*}
\int_{\omega}\sigma_{ij}\left(\varphi_{i}-v_{i}\right)n_{j}dx'
=\int_{\omega}\sigma_{t_{i}}\left(\varphi_{i}-v_{i}\right)dx'+\int_{\omega}
\sigma_{n}n_{i}\left(\varphi_{i}-v_{i}\right)dx'.
\end{eqnarray*}
Knowing that from (\ref{7}), $n_{i}\left(\varphi_{i}-v_{i}\right)=0$ on
$\omega$, then
\begin{eqnarray*}
\int_{\omega}\sigma_{ij}\left(\varphi_{i}-v_{i}\right)n_{j}dx'
=\int_{\omega}\sigma_{t_{i}}\left(\varphi_{i}-v_{i}\right)dx'
\end{eqnarray*}
and (\ref{17}) becomes
\begin{eqnarray}\label{19}
\int_{\Omega}\sigma_{ij}\partial_{j}\left(\varphi_{i}-v_{i}\right)dx
=\int_{\omega}\sigma_{t_{i}}\left(\varphi_{i}-v_{i}\right)dx'
+\int_{\Omega}f_{i}\left(\varphi_{i}-v_{i}\right)dx.
\end{eqnarray}
Let us involve the Tresca condition.
For this end, we add to both sides of (\ref{19}) the term
$\int_{\omega}k\left(\left|\varphi-s\right|-\left|v-s\right|\right)dx'$, then
\begin{eqnarray}\label{20}
\int_{\Omega}\sigma_{ij}\partial_{j}\left(\varphi_{i}-v_{i}\right)dx
+\int_{\omega}k\left(\left|\varphi-s\right|-\left|v-s\right|\right)dx'
%\nonumber\\&=&
=\int_{\Omega}f_{i}\left(\varphi_{i}-v_{i}\right)dx+A,
\end{eqnarray}
with
\begin{eqnarray*}
A=\int_{\omega}\left[\sigma_{t_{i}}\left(\varphi_{i}-v_{i}
\right)+k\left(\left|\varphi-s\right|-\left|v-s\right|\right)\right]dx'.
\end{eqnarray*}
Let us prove that $A$ is positive. First, Following \cite{dlions} Chap.3 page
140, we prove that the condition (\ref{8})
is equivalent to
\begin{eqnarray}\label{tc2}
(v_t-s).\sigma_{t}+k|v_t-s|=0,
\end{eqnarray}
indeed if (\ref{8}) holds and $|\sigma_{t}|=k$, then $v_t=s-\lambda \sigma_{t}$
for some$\lambda \geq 0$, so
\begin{eqnarray*}
(v_t-s).\sigma_{t}+k|v_t-s|=-\sigma_{t}.\sigma_{t}+k\lambda|\sigma_{t}|
=-\lambda\sigma_{t}^2+\lambda\sigma_{t}^2=0.
\end{eqnarray*}
Now if $|\sigma_{t}|<k$, by (\ref{8}) $v_t=s$ and (\ref{tc2}) holds.\newline
Conversely if $|\sigma_{t}|=k$, then by (\ref{tc2})
$(v_t-s).\sigma_{t}=-|\sigma_{t}||v_t-s|$, so there exists $\lambda \geq 0$ such
that
$v_t-s=-\lambda \sigma_{t}$, that is $v_t=s-\lambda \sigma_{t}$. The first part
of (\ref{8}) is shown. If $|\sigma_{t}|<k$, then by (\ref{tc2}) we have
$(v_t-s).\sigma_{t}+k|v_t-s|=0 \geq |v_t-s|(k-|\sigma_{t}|)$,
 thus $v_t-s=0$ because $|\sigma_{t}|<k$. The second part of (\ref{8}) holds,
and the assertion is proved.\newline Now by (\ref{7})
we deduce that $v=v_{t}$ on $\omega$, hence
\begin{eqnarray*}
A=\int_{\omega}(\sigma_{t}.(\varphi-s)+k|\varphi-s|)dx'.
\end{eqnarray*}
But $\sigma_{t}.(\varphi-s)\geq-|\sigma_{t}||\varphi-s|$, and since
$|\sigma_{t}|\leq k$ on $\omega$, it follows that
\begin{eqnarray*}
\sigma_{t}.(\varphi-s)+k|\varphi-s|\geq0\quad on \quad \omega,
\end{eqnarray*}
this shows that $A$ is positive. And (\ref{20}) becomes
\begin{eqnarray*}
\int_{\Omega}\sigma_{ij}\partial_{j}\left(\varphi_{i}-v_{i}\right)
dx+\int_{\omega}k\left(\left|\varphi-s\right|-\left|v-s\right|\right)dx'\geq
\int_{\Omega}f_{i}\left(\varphi_{i}-v_{i}\right)dx.
\end{eqnarray*}
Replacing $\sigma_{ij}$ by its expression (\ref{9}) and using (\ref{3}), we
obtain the variationnal inequality for the velocity field $v$. For all
$\varphi\in V^q$
\begin{eqnarray*}
\int_{\Omega}2\mu(\theta,v,|D(v)|)d_{ij}(v)\partial_{j}\left(\varphi_{i}-v_{i}
\right) dx
-\int_{\Omega}\pi div(\varphi)dx \nonumber\\
+\int_{\omega}k\left(\left|\varphi-s\right|-\left|v-s\right|\right)dx'\geq
\int_{\Omega}f_{i}\left(\varphi_{i}-v_{i}\right)dx.
\end{eqnarray*}
Similarly, by multiplying (\ref{2}) by $\psi\in
W^{1,p}_{\Gamma_{1}\cup\Gamma_{L}}\left(\Omega\right)$ we obtain (\ref{15}).
\end{proof}

Subsequently, we will use the following well known results,
\begin{eqnarray}\label{21}
\int_{\Omega}\left|D(u)\right|^2dx\leq\int_{\Omega}\left|\nabla u\right|^2dx,
\qquad
\forall u\in \left(W^{1,p}(\Omega)\right)^{3},
\end{eqnarray}
 \begin{eqnarray}\label{22}
\int_{\Omega}\left|D(u)\right|^2dx\geq \frac{1}{2}\int_{\Omega}\left|\nabla
u\right|^2dx, \qquad \forall u\in V_{0}^{p}.
\end{eqnarray}

\renewcommand{\theequation}{3.\arabic{equation}}
\setcounter{equation}{0}
\section{Existence results}\label{sec3}
 We assume that the function $\mu$ defined on $\R\times\R^3\times\R_{+}$ is such
that
\begin{eqnarray}\label{23}
\exists\mu_{0},\mu_{1}\in\R;\quad 0<\mu_{0}\leq\mu(t,u,s)\leq\mu_{1},\quad
\forall \left(t,u,s\right)\in\R\times\R^3\times\R_{+},
\end{eqnarray}
\begin{eqnarray}\label{24}
the \quad function \quad s\mapsto\mu(.,.,s)\quad is \quad monotone \quad on
\quad\R_{+}.
\end{eqnarray}

\subsection{First intermediate problem}\label{sec3.1}
From the variational inequality (\ref{14}), we obtain the following intermediate
problem

\begin{prob}\label{prob3.1}
For given $\theta \in W^{1,q}_{\Gamma_{1}\cup\Gamma_{L}}(\Omega)$ and
$f\in(W^{1,p}(\Omega))^3$,
we look for $v\in V_{div}^{p}$ satisfying the following variational inequality
\begin{eqnarray}\label{25}
a\left(\theta, v,
\varphi-v\right)+j\left(\varphi\right)-j\left(v\right)\geq\left(f,
\varphi-v\right),\quad \forall \varphi \in V_{div}^{q}.
\end{eqnarray}
\end{prob}
To solve this problem we will use the Nonlinear Operators Theory \cite{J.L}.
 We define the operator
\begin{eqnarray}\label{A}
A:V^p_{\Gamma_{1}}\rightarrow(V^p_{\Gamma_{1}})'\quad
by \quad \left\langle A(v),\varphi\right\rangle=a(\theta,v,\varphi),
\end{eqnarray}
where $\left\langle ., .\right\rangle$ is the duality brackets between
$(V^p_{\Gamma_{1}})'$ and $V^p_{\Gamma_{1}}$,
and we denote by $\Lambda_{V_{div}}$ the characteristic function of $V^p_{div}$,
\begin{eqnarray*}
\Lambda_{V_{div}}(u)=0\quad if\quad u\in V^p_{div},\quad \mbox{and}\quad
\Lambda_{V_{div}}(u)=+\infty\quad if \quad u\notin V^p_{div}.
\end{eqnarray*}
Then the variational inequality (\ref{25}) becomes,
\begin{eqnarray}\label{26}
\left\langle A(v), \varphi-v
\right\rangle+j\left(\varphi\right)+\Lambda_{V_{div}}(\varphi)
-j\left(v\right)-\Lambda_{V_{div}}(v)\geq\left(f,\varphi-v\right), \forall
\varphi\in V^q_{\Gamma_{1}}.
\end{eqnarray}

\begin{lemma}\label{lem3.1}
The operator $A$ defined by (\ref{A}) is bounded, hemicontinuous and monotone
on $V^p_{\Gamma_{1}}$.
\end{lemma}
\begin{proof}
For any $v\in V^p_{\Gamma_{1}}$ and $\varphi\in V^q_{\Gamma_{1}}$, we have
\begin{eqnarray*}
\left|\left\langle A(v),\varphi\right\rangle \right|
=\left|
\int_{\Omega}2\mu\left(\theta,v,|D(v)|\right)d_{ij}(v)\partial_{j}\varphi_{i} dx
\right|
\leq 2\mu_{1}\left|\int_{\Omega} d_{ij}(v)d_{ij}( \varphi ) dx\right|,
\end{eqnarray*}
using H$\ddot{\mbox {o}}$lder's and Minkowski's inequalities, and (\ref{23}), we
get
\begin{eqnarray*}
\left|\left\langle A(v),\varphi\right\rangle\right|
&\leq& 2\mu_{1}\int_{\Omega}\left(\sum^{3}_{i,j=1}
\left|d_{ij}(v)\right|^p \right)^\frac{1}{p} \left(\sum^{3}_{i,j=1}\left|d_{ij}
\left(\varphi\right)\right|^q \right)^\frac{1}{q}dx\\
&\leq&2\mu_{1}\int_{\Omega}\left(\sum^{3}_{i,j=1}\left|\partial_{j}
v_{i}\right|^p \right)^\frac{1}{p}
\left(\sum^{3}_{i,j=1}\left|\partial_{j}\varphi_{i}\right|^q
\right)^\frac{1}{q}dx
\end{eqnarray*}

so
\begin{eqnarray*}
\left|\left\langle A(v),\varphi\right\rangle\right|
\leq 2\mu_{1}\left(\int_{\Omega}\left|\nabla v\right|^p
dx\right)^\frac{1}{p}\left(\int_{\Omega}\left|\nabla\varphi\right|^q
dx\right)^\frac{1}{q}
\leq 2\mu_{1}\left\|v\right\|_{1,p}\left\|\varphi\right\|_{1,q} \quad\forall
\left(v,\varphi\right) \in V^p_{\Gamma_{1}}\times V^q_{\Gamma_{1}},
\end{eqnarray*}
then $A$ is bounded.\newline\\
We show that $A$ is hemicontinuous. For this, we prove that for any $u,v,w\in
V^p_{\Gamma_{1}}$, the function
\begin{eqnarray*}
\alpha:\R\rightarrow\R,\quad \alpha(t)=\left\langle A(u+tv),w\right\rangle
\end{eqnarray*}
is continuous.
We have
\begin{eqnarray*}
\alpha(t)=\int_{\Omega}2\mu\left(\theta,u+tv,
\left|D(u+tv)\right|\right)D\left(u+tv\right):D\left(w\right)dx.
\end{eqnarray*}
The function
\begin{eqnarray*}
t\mapsto
s(t)=2\mu\left(\theta,u+tv,
\left|D(u+tv)\right|\right)D\left(u+tv\right):D\left(w\right)
\end{eqnarray*}
is obviously continuous on $\R$. Let $(t_{n})$ be a sequence converging to $t$
in $\R$. Then $s_{n}:=s(t_{n})\in L^1(\Omega)$ and $(s_{n})$
 converges to $s(t)$ when $n$ goes to $+\infty$.\newline
The sequence $(t_{n})$ is bounded, so there exists $M>0$ such that
$\left|t_{n}\right|<M, \forall n\geq0$.\newline
Then we have
\begin{eqnarray*}
\left|s_{n}\right|\leq
g:=2\mu_{1}
\left(\left|D\left(u\right)\right|\left|D\left(w\right)\right| +
M\left|D\left(v\right)\right|
\left|D\left(w\right)\right|\right).
\end{eqnarray*}
Since $g\in L^1(\Omega)$ and is positive, by Dominated Convergence Theorem, we
deduce that $s(t)\in L^1(\Omega)$ and
\begin{eqnarray*}
\lim_{n\longrightarrow+\infty}\int_{\Omega}s_{n}dx=\int_{\Omega}s(t)dx,
\end{eqnarray*}
that is
\begin{eqnarray*}
\lim_{n\longrightarrow+\infty}\alpha\left(t_{n}\right)=\alpha\left(t\right).
\end{eqnarray*}
This shows that $A$ is hemicontinuous.
For the monotonicity of $A$, we establish that,
\begin{eqnarray*}
\left\langle A\left(u\right)-A\left(v\right),u-v\right\rangle\geq0,\quad
\forall u,v\in V^p_{\Gamma_{1}}.
\end{eqnarray*} We have,
\begin{eqnarray*}
\left\langle A\left(u\right)-A\left(v\right),u-v\right\rangle
=2\int_{\Omega}\left[\mu\left(\theta,u,\left|D\left(u\right)\right|\right)d_{ij}
\left(u\right)-\mu\left(\theta,v,\left|D\left(v\right)\right|\right)d_{ij}
\left(v\right)\right]\partial_{j}\left(u_{i}-v_{i}\right)dx\\
\end{eqnarray*}
we have also
\begin{eqnarray*}
\left\langle A\left(u\right)-A\left(v\right),u-v\right\rangle
&=&\int_{\Omega}2\mu\left(\theta,u,\left|D\left(u\right)\right|\right)d_{ij}
\left(u\right)\partial_{j}u_{i}dx+\int_{\Omega}2\mu\left(\theta,v,
\left|D\left(v\right)\right|\right)d_{ij}\left(v\right)\partial_{j}v_{i}dx
\\&-&\int_{\Omega}2\mu\left(\theta,u,\left|D\left(u\right)\right|\right)d_{ij}
\left(u\right)\partial_{j}v_{i}dx-\int_{\Omega}2\mu\left(\theta,v,
\left|D\left(v\right)\right|\right)d_{ij}\left(v\right)\partial_{j}u_{i}dx.
\end{eqnarray*}
By the fact that
\begin{eqnarray*}
d_{ij}\left(u\right)\partial_{j}v_{i}=d_{ij}\left(u\right)d_{ij}
\left(v\right)=D\left(u\right):D\left(v\right)\leq\left|D\left(u\right)
\right|\left|D\left(v\right)\right|
\end{eqnarray*}
 we obtain
\begin{eqnarray*}
\left\langle
A\left(u\right)-A\left(v\right),u-v\right\rangle&\geq&\int_{\Omega}
2\mu\left(\theta,u,
\left|D\left(u\right)\right|\right)\left|D\left(u\right)\right|^2dx+\int_{\Omega
}2\mu\left(\theta,v,
\left|D\left(v\right)\right|\right)\left|D\left(v\right)\right|^2dx\\&-&2\int_{
\Omega}\left(\mu\left(\theta,u,
\left|D\left(u\right)\right|\right)+\mu\left(\theta,v,
\left|D\left(v\right)\right|\right)\right)\left|D\left(u\right)\right|
\left|D\left(v\right)\right|dx.
\end{eqnarray*}
Now since,
$2\left|D\left(u\right)\right|\left|D\left(v\right)\right|\leq\left|D\left(u
\right)\right|^2+\left|D\left(v\right)\right|^2$
we get
\begin{eqnarray*}
\left\langle
A\left(u\right)-A\left(v\right),u-v\right\rangle\geq\int_{\Omega}
\left(\mu\left(\theta,u,\left|D\left(u\right)\right|\right)-\mu\left(\theta,v,
\left|D\left(v\right)\right|\right)\right)\left(\left|D\left(u\right)\right|^2-
\left|D\left(v\right)\right|^2\right)dx.
\end{eqnarray*}
Using (\ref{24}), we deduce that
\begin{eqnarray*}
\left\langle A\left(u\right)-A\left(v\right),u-v\right\rangle\geq0\quad \forall
u, v\in u,v\in V^p_{\Gamma_{1}},
\end{eqnarray*}
That is $A$ is monotone.
\end{proof}

Now we are in position to give the existence Theorem for Problem \ref{prob3.1}.
\begin{theorem}\label{th3.1}
Assume that (\ref{23}) and (\ref{24}) are satisfied, $f\in(W^{1,p}(\Omega))^3$
and $0\leq k\in L^p(\omega)$, then for fixed
 $\theta$ in
$W^{1,q}_{\Gamma_{1}\cup\Gamma_{L}}\left(\Omega\right)$ there exists
$v_{\theta}\in V^p_{div}$ solution of (\ref{25}).
\end{theorem}
\begin{proof}
From Lemma \ref{lem3.1}, we conclude that the operator $A$ is pseudo-monotone.
The function
 $v \mapsto j\left(v\right)+\Lambda_{V_{div}}(v)$ is convex, proper and lower
semi-continuous. Let us check the coercivity condition, that is
\begin{eqnarray*}
\exists v^*\in V^p_{\Gamma_{1}}\quad such \quad that\quad
j\left(v^*\right)+\Lambda_{V_{div}}(v^*)<+\infty
\end{eqnarray*}
and
\begin{eqnarray*}
\lim_{\left\|v\right\|_{1,p}\rightarrow +\infty} \frac{\left\langle
A(v),v-v^*\right\rangle+j\left(v\right)
+\Lambda_{V_{div}}(v)}{\left\|v\right\|_{1,p}}=+\infty.
\end{eqnarray*}
We choose $v^*=G$, and observe, from their explicit expressions, that
\begin{eqnarray}\label{ac}
\mbox{if} \quad\lim_{\|v\|_{1,p}\rightarrow\infty} \frac{\langle A(v-G),v-G
\rangle}{\|v\|_{1,p}} = +\infty
\quad
\mbox{then} \quad \lim_{\|v\|_{1,p}\rightarrow\infty}\frac{\langle A(v),v-G
\rangle}{\|v\|_{1,p}}= +\infty.
\end{eqnarray}
Indeed
\begin{eqnarray*}
\left\langle A(v-G),v-G\right\rangle&=&
2\int_{\Omega}\mu(\theta,v-G,
|D(v-G)|)\left(\left|D(v)\right|^2-2D(v):D(G)\right)dx
\\&+&2\int_{\Omega}\mu(\theta,v-G,|D(v-G)|)\left|D(G)\right|^2dx,
\end{eqnarray*}
and
\begin{eqnarray*}
\left\langle
A(v),v-G\right\rangle=2\int_{\Omega}\mu(\theta,v,
|D(v)|)\left(\left|D(v)\right|^2-D(v):D(G)\right)dx.
\end{eqnarray*}
Now since the function $\mu$ is positive and bounded on $\R\times\R^3\times\R_+$
and  $\left\|G\right\|_{1,2}^2 \left\|v\right\|^{-1}_{1,p}\rightarrow0$ when
$\left\|v\right\|_{1,p}\rightarrow\infty$, (\ref{ac}) follows.
Then we can use here $\left\langle A(v-G),v-G\right\rangle$ instead of
$\left\langle A(v),v-G\right\rangle$. For $v\in V^p_{\Gamma_1}$, by positivity
of the function $j+\Lambda_{V_{div}}$ and (\ref{23}) we obtain
\begin{eqnarray*}
\left\langle A(v-G),v-G\right\rangle+j\left(v\right)+\Lambda_{V_{div}}(v)
\geq2\mu_0\int_{\Omega}|D(v-G)|^2dx.
\end{eqnarray*}
We have $v-G\in V_0^{p}$ then by (\ref{22}), we get
\begin{eqnarray*}
\left\langle A(v-G),v-G\right\rangle+j\left(v\right)+\Lambda_{V_{div}}(v)
\geq \mu_{0}\left\|v-G\right\|^2_{1,2}
 \end{eqnarray*}
and by Cauchy-Schwarz inequality
\begin{eqnarray*}
 \mu_{0}\left\|v-G\right\|^2_{1,2}
\geq\mu_{0}\left\|v\right\|^2_{1,2}-2\mu_0\left\|v\right\|_{1,2}
\left\|G\right\|_{1,2}
+\mu_0\left\|G\right\|^2_{1,2},
\end{eqnarray*}
so, we obtain
\begin{eqnarray}\label{27}
\frac{\left\langle A(v-G), v-G\right\rangle +
j(v)+\Lambda_{V_{div}}(v)}{\|v\|_{1,p}}
\geq
\mu_{0}\left(\|v\|_{1,2}\frac{\|v\|_{1,2}}{\|v\|_{1,p}}
-2 \|G\|_{1,2}\frac{\|v\|_{1,2}}{\|v\|_{1,p}}
+\frac{\|G\|^2_{1,2}}{\|v\|_{1,p}}\right).
\end{eqnarray}

By the continuity of the embedding $V^p_{\Gamma_{1}}\subset V^2_{\Gamma_{1}}$,
we deduce that there exists a constant $c>0$ depending only on $\Omega$ and $p$
such that $\left\|v\right\|_{1,2}\leq c\left\|v\right\|_{1,p}$. Letting
$\left\|v\right\|_{1,2}\rightarrow\infty$,  then
$\left\|v\right\|_{1,p}\rightarrow\infty$ and we have
\begin{eqnarray}\label{28}
\lim_{\left\|v\right\|_{1,2}\rightarrow +\infty}\frac{\left\|v\right\|_{1,2}}
{\left\|v\right\|_{1,p}}\neq 0,
\end{eqnarray}
indeed, for $v\in V^p_{\Gamma_{1}}$ and $v\neq 0$, set $w=v
\left\|v\right\|_{1,p}^{-1}$ then
\begin{eqnarray}\label{29}
\left\|w\right\|_{1,p}=1.
\end{eqnarray}
But
\begin{eqnarray*}
\lim_{\left\|v\right\|_{1,2}\rightarrow
+\infty}\frac{\left\|v\right\|_{1,2}}{\left\|v\right\|_{1,p}}=0
\Longrightarrow
\lim_{\left\|v\right\|_{1,2}\rightarrow +\infty}\left\|w\right\|_{1,2}=0.
\end{eqnarray*}
This is a contradition with (\ref{29}), so (\ref{28}) holds.
Now by letting $\left\|v\right\|_{1,2}\rightarrow\infty$,
we get $\left\|v\right\|_{1,p}\rightarrow\infty$ and  we can deduce, by
(\ref{28}),
that the right side hand of (\ref{27}) tends to $+\infty$.
consequently
\begin{eqnarray*}
\frac{\left\langle
A(v-G),v-G\right\rangle+j\left(v\right)+\Lambda_{V_{div}}(v)}{\left\|v\right\|_{
1,p}}\rightarrow\ +\infty.
\end{eqnarray*}
So by (\ref{ac}),
\begin{eqnarray*}
\frac{\left\langle
A(v),v-G\right\rangle+j\left(v\right)+\Lambda_{V_{div}}(v)}{\left\|v\right\|_{1,
p}}\rightarrow\ +\infty.
\end{eqnarray*}
This shows that the coercivity condition is satisfied.
Applying Theorem 8.5 chap. 2 in \cite{J.L}, we conclude that (\ref{26}) and
hence (\ref{25}) admits a solution $v_{\theta}$ in the space $V^p_{div}$.
\end{proof}

Before stating the existence Theorem of the pressure we need to prove the
following lemma.

\begin{lemma}\label{lem3.2}
The solution $v_{\theta}$ of Problem 3.1 is bounded in $V^p_{div}$ independently
of the temperature $\theta$.
\end{lemma}
\begin{proof}
As $v_{\theta}$ satisties the following variational inequality
\begin{eqnarray*}
a\left(\theta, v_{\theta}, \varphi-
v_{\theta}\right)+j\left(\varphi\right)-j\left(v_{\theta}\right)\geq\left(f,
\varphi-v_{\theta}\right),\quad \forall \varphi \in V_{div}^{q},
\end{eqnarray*}
by taking $\varphi=G\in V_{div}^{p}\subset V_{div}^{q}$, we get
\begin{eqnarray}\label{rm}
\left\langle
A(v_{\theta}),v_{\theta}-G\right\rangle\leq\left(f,v_{\theta}\right)-\left(f,
G\right)+j\left(G\right),
\end{eqnarray}
because $j$ is positive. Remark that we can write
\begin{eqnarray*}
\left\langle A(v_{\theta}),v_{\theta}-G\right\rangle
&=&2\int_{\Omega}\mu(\theta,v_{\theta},|D(v_{\theta})|)d_{ij}(v_{\theta}
-G)\partial_j(v^{i}_{\theta}-G_i)dx\\&+&2\int_{\Omega}\mu(\theta,v_{\theta},
|D(v_{\theta})|)d_{ij}(G)\partial_j(v^{i}_{\theta}-G_i)dx,
\end{eqnarray*}
where $v^{i}_{\theta}$ is the i-th component of $v_{\theta}$. Then
\begin{eqnarray}\label{im}
\left\langle
A(v_{\theta}),v_{\theta}-G\right\rangle&=&2\int_{\Omega}\mu(\theta,v_{\theta},
|D(v_{\theta})|)\left(\left|D(v_{\theta}-G)\right|^2+D(G):D(v_{\theta}
)\right)dx\nonumber\\&-&2\int_{\Omega}\mu(\theta,v_{\theta},|D(v_{\theta}
)|)\left|D(G)\right|^2dx.
\end{eqnarray}
Now from  (\ref{23}),  (\ref{rm}),  (\ref{im}),
(\ref{22}), (\ref{21}) and H$\ddot{\mbox{o}}$lder's inequality we get
\begin{eqnarray*}
\mu_0\left\|v_{\theta}-G\right\|_{1,2}^2&\leq&2\mu_1\left\|v_\theta\right\|_{1,2
}\left\|G\right\|_{1,2}+\left\|f\right\|_{W^{1,p}}\left\|v_\theta\right\|_{1,q}
\\&+&2\mu_{1}\left\|G\right\|^2_{1,2}+\left\|f\right\|_{W^{1,p}}
\left\|G\right\|_{1,q}+j(G),
\end{eqnarray*}
hence
\begin{eqnarray}\label{30}
\mu_{0}\left\|v_\theta\right\|^2_{1,2}\leq2(\mu_{0}+\mu_{1}
)\left\|v_\theta\right\|_{1,2}\left\|G\right\|_{1,2}&+&\left\|f\right\|_{W^{1,p}
}\left\|v_\theta\right\|_{1,q}+2\mu_{1}\left\|G\right\|^2_{1,2}
\nonumber\\&+&\left\|f\right\|_{W^{1,p}}\left\|G\right\|_{1,q}+j(G).
\end{eqnarray}
From the continuous embedding $V_{\Gamma_1}^{2}\subset V_{\Gamma_1}^{q}$, there
exists a positive constant $\beta$ such that,
\begin{eqnarray*}
\left\|v_\theta\right\|_{1,q}\leq \beta \left\|v_\theta\right\|_{1,2}.
\end{eqnarray*}
 Then (\ref{30}) becomes
\begin{eqnarray}\label{31}
\mu_{0}\left\|v_\theta\right\|^2_{1,2}\leq2(\mu_{0}+\mu_{1}
)\left\|v_\theta\right\|_{1,2}\left\|G\right\|_{1,2}&+&\beta\left\|f\right\|_{W^
{1,p}}\left\|v_\theta\right\|_{1,2}+2\mu_{1}\left\|G\right\|^2_{1,2}
\nonumber\\&+&\left\|f\right\|_{W^{1,p}}\left\|G\right\|_{1,q}+j(G).
\end{eqnarray}
By (\ref{31}) we deduce that there exists a positive constant $C$ independent of
$\theta$ such that
\begin{eqnarray}\label{32}
\left\|v_\theta\right\|_{1,2}\leq C.
\end{eqnarray}
 Indeed, on the contrary, by dividing the two sides of (\ref{31}) by
$\left\|v_\theta\right\|^2_{1,2}$ and letting
$\left\|v_\theta\right\|_{1,2}\rightarrow +\infty$ we obtain
%\begin{eqnarray*}
$\mu_{0}\leq 0$.
  %\end{eqnarray*}
This is a contradiction because $\mu_{0}>0$, then (\ref{32}) holds.\newline\\
By (\ref{28}), $\left\|v_\theta\right\|^{-1}_{1,2}\left\|v_\theta\right\|_{1,p}$
is bounded for $\left\|v_\theta\right\|_{1,2}$ large enough, so it follows from
(\ref{32}), the existence of a positive constant $C'$ independent of $\theta$
such that
\begin{eqnarray}\label{33}
\left\|v_\theta\right\|_{1,p}\leq C'.
\end{eqnarray}
 The lemma is proved.
\end{proof}

\begin{theorem}\label{th3.2}
Under the assumptions of {\rm Theorem \ref{th3.1}}, there exists a unique
$\pi\in L^p_0\left(\Omega\right)$ satisfying equation (\ref{14}).
\end{theorem}
\begin{proof}
Let $v_{\theta}$ be the solution of (\ref{25}).
By taking $\varphi=v_{\theta}\pm\phi$, for all $\phi\in V^q_{0,div}$, we obtain
from (\ref{25}) the following variational equation,
\begin{eqnarray}\label{34}
a\left(\theta,v_{\theta},\phi\right)=\left(f,\phi\right),\quad \forall \phi\in
V^q_{0,div}.
\end{eqnarray}
Consider the linear form $F$ defined on $V^{q}_{0}$ by
\begin{eqnarray*}
F\left(\phi\right)=a\left(\theta,v_{\theta},\phi\right)-\left(f,\phi\right).
\end{eqnarray*}
We prove that $F$ is continuous on $V^{q}_0$.
For all $\phi$ in $V^{q}_0$ we have
\begin{eqnarray*}
|F(\phi)| \leq 2
|\Omega|^{\frac{p-2}{p}}\mu_{1}\|v_{\theta}\|_{1,p}\|\phi\|_{1,p}+|\Omega|^{
\frac{p-2}{p}}\|f\|_{1,p}\|\phi\|_{1,q},
\end{eqnarray*}
by Lemma \ref{lem3.2} we have (\ref{33}) so
\begin{eqnarray*}
 |F(\phi)| \leq  |\Omega|^{\frac{p-2}{p}}\left(
2\mu_{1}C'+\left\|f\right\|_{1,p}\right)\left\|\phi\right\|_{1,q}.
\end{eqnarray*}
This shows the continuity of $F$, and that $F\in W^{-1, p}(\Omega)$. And since
\begin{eqnarray*}
F\left(\phi\right)=0, \quad \forall\phi\in V^q_{0,div},
\end{eqnarray*}
by De Rham's Theorem in \cite{Amrgir} page 116, we deduce the existence of a
unique $\pi\in L^p_{0}\left(\Omega\right)$, such that
\begin{eqnarray*}
F\left(\phi\right)=\left\langle \nabla\pi,\phi\right\rangle \quad \forall\phi\in
V^q_{0},
\end{eqnarray*}
using Green's formula we deduce that
\begin{eqnarray*}
2div\left(\mu(\theta,v_{\theta},|D(v_{\theta})|\right)D(v_{\theta})+f=\nabla
\pi.
\end{eqnarray*}
By multiplying this equality by $\varphi\in V^q$ and using Green's formula
again, we deduce that $\left(v_{\theta},\pi\right)\in V_{div}^p\times
L^p_{0}\left(\Omega\right)$ satisfies (\ref{14}).
\end{proof}

\subsection{Second intermediate problem}\label{sec3.2}
Recall that the temperature satisfies the variational equation (\ref{15}) with
(\ref{L}) and (\ref{B}).
We assume that the function $K$ is  also such that
\begin{eqnarray}\label{36}
\exists k_0, k_1 \in\R; \quad 0< k_0\leq K(x) \leq k_1, \quad \forall x\in
\Omega.
\end{eqnarray}
In fact we must take $p\geq 4$ and seek the solution $\theta$ in the subspace
$H^{1}_{\Gamma_{1}\cup\Gamma_{L}}\left(\Omega\right)$ of
$W^{1,q}_{\Gamma_{1}\cup\Gamma_{L}}\left(\Omega\right)$,
this choice will be justified below in the proof of Theorem 3.4.
Let us consider the second intermediate problem.
\begin{prob}\label{Sprob}
For given $v\in V^{p}_{div}$, find $\theta\in
H^{1}_{\Gamma_{1}\cup\Gamma_{L}}\left(\Omega\right)$
solution of the equation,
 \begin{eqnarray}\label{sp}
B\left(\theta,\psi\right)=L\left( \theta, \psi\right),
\quad \forall\psi \in H^{1}_{\Gamma_{1}\cup\Gamma_{L}}\left(\Omega\right).
\end{eqnarray}
\end{prob}
Remark that, from {\rm(\ref{L})}, $L$ depends on $v$. To study this nonlinear
problem we consider first the following corresponding linearized problem.
\begin{prob}\label{pro3.2}
For given $v\in V^p_{div}$ and $\eta\in
H^{1}_{\Gamma_{1}\cup\Gamma_{L}}\left(\Omega\right)$, find
$\theta\in H^{1}_{\Gamma_{1}\cup\Gamma_{L}}\left(\Omega\right)$ solution of the
equation,
\begin{eqnarray}\label{37}
B\left(\theta,\psi\right)=L\left( \eta, \psi\right),\quad \forall\psi \in
H^{1}_{\Gamma_{1}\cup\Gamma_{L}}\left(\Omega\right).
\end{eqnarray}
\end{prob}

\begin{theorem}\label{thm3.3}
Assume that $p\geq 4$, (\ref{36}) and the assumptions of {\rm Theorem
\ref{th3.1}} hold.
 Then for $\eta\in H^{1}_{\Gamma_{1}\cup\Gamma_{L}}\left(\Omega\right), v\in
V^p_{div}, \theta_\omega \in L^2\left(\omega\right)$ and $r\in
L^\infty\left(\R\right)$, there exists a unique $\theta \in
H^{1}_{\Gamma_{1}\cup\Gamma_{L}}\left(\Omega\right)$ solution for (\ref{37}).
\end{theorem}
\begin{proof}
The bilinear form $B$ is continuous on
$H^{1}_{\Gamma_{1}\cup\Gamma_{L}}\left(\Omega\right)\times
H^{1}_{\Gamma_{1}\cup\Gamma_{L}}\left(\Omega\right)$. Indeed, by (\ref{36}) and
since $v\in V^p_{div}\subset \left(L^\infty\left(\Omega\right)\right)^3$, there
exists $M>0$ such that
\begin{eqnarray*}
\left|B\left(\theta,\psi\right)\right|\leq
k_1\left\|\theta\right\|_{1,2}\left\|\psi\right\|_{1,2}+M\left\|\theta\right\|_{
1,2}\left\|\psi\right\|_2.
\end{eqnarray*}
Using the Poincar\'e's inequality we get, for some positive constant $C$,
\begin{eqnarray*}
\left|B\left(\theta,\psi\right)\right|\leq(k_1+MC)\left\|\theta\right\|_{1,2}
\left\|\psi\right\|_{1,2}.
\end{eqnarray*}
This shows the continuity of the form $B$.
Let us prove that $B$ is coercive. We have, for any $\theta\in
H^{1}_{\Gamma_{1}\cup\Gamma_{L}}\left(\Omega\right)$,
\begin{eqnarray*}
B\left(\theta,\theta\right)=\int_\Omega
K(x)\left|\nabla\theta\right|^2dx+\int_\Omega\theta v_i \partial_i \theta dx
\geq
k_0\int_\Omega\left|\nabla\theta\right|^2dx+\int_\Omega\frac{1}{2}
v_i\partial_i(\theta^2)dx.
\end{eqnarray*}
 By Green's formula, we obtain
\begin{eqnarray}\label{38}
\vartheta:=\int_\Omega v_i\partial_i(\theta^2) dx=\int_{\partial\Omega}\theta^2
v_i n_i ds-\int_\Omega\theta^2\partial_iv_idx.
\end{eqnarray}
Since $\theta=0$ on $\Gamma_{1}\cup\Gamma_{L}$, $n.v=0$ on $\omega$ and
$div(v)=0$ in $\Omega$ we deduce that
$\vartheta=0$, and then
\begin{eqnarray*}
B\left(\theta,\theta\right)\geq k_0\left\|\theta\right\|^2_{1,2}, \quad
\forall\theta\in H^{1}_{\Gamma_{1}\cup\Gamma_{L}}\left(\Omega\right).
\end{eqnarray*}
The coercivity of $B$ follows. We prove that the linear form $L(\eta, .)$ is
continuous on $H^{1}_{\Gamma_{1}\cup\Gamma_{L}}\left(\Omega\right)$.
 Let $\psi\in H^{1}_{\Gamma_{1}\cup\Gamma_{L}}\left(\Omega\right)$,
by the continuous embedding
$H^{1}_{\Gamma_{1}\cup\Gamma_{L}}\left(\Omega\right)\subset L^2(\omega)$ we
obtain
\begin{eqnarray*}
|L\left(\eta, \psi\right)| \leq
2\mu_1\int_\Omega\left|D\left(v\right)\right|^2\left|\psi\right|dx +
C\int_\Omega\left|\psi\right|dx
+\int_\omega\left|\theta_\omega\right|\left|\psi\right|dx',
\end{eqnarray*}
and by Poincar\'e's and H$\ddot{\mbox o}$lder's inequalities, we obtain
\begin{eqnarray*}
\left|L\left(\eta,
\psi\right)\right|\leq\left[2\mu_1C_1\left|\Omega\right|^\frac{p-4}{2p}
\left\|D(v)\right\|^2_{p}+C_2
\left|\Omega\right|^\frac{1}{2}+C_3\left\|\theta_\omega\right\|_{L^2(\omega)}
\right]\left\|\psi\right\|_{1,2}.
\end{eqnarray*}
As $p\geq 4$, this proves the continuity of $L(\eta,.)$, and by Lax-Milgram
Theorem, we deduce that
there exists a unique $\theta\in
H^{1}_{\Gamma_{1}\cup\Gamma_{L}}\left(\Omega\right)$
solution of the linearized problem \ref{pro3.2}.
\end{proof}

In the following theorem we prove only the existence of at least one solution to
the intermediate {\rm Problem \ref{Sprob}}.

\begin{theorem}\label{thm3.4}
Let $\theta_\omega \in L^2\left(\omega\right)$ and $r \in
L^\infty\left(\R\right)$.
Assume that functions $r$ and $t\mapsto\mu (t, ., .)$ are Lipschitzian.
Then with the same assumptions as {\rm Theorem \ref{thm3.3}}, there exists at
least one
$\theta \in H^{1}_{\Gamma_{1}\cup\Gamma_{L}}\left(\Omega\right)$ solution to the
intermediate {\rm Problem \ref{Sprob}}.
\end{theorem}
\begin{proof}
For $\eta\in H^{1}_{\Gamma_{1}\cup\Gamma_{L}}\left(\Omega\right)$, Theorem
\ref{thm3.3} ensures the existence and the uniqueness of $\theta\in
H^{1}_{\Gamma_{1}\cup\Gamma_{L}}\left(\Omega\right)$ solution of the linearized
problem (\ref{37}). Then we can define the operator
\begin{eqnarray}\label{opT}
T:
H^{1}_{\Gamma_{1}\cup\Gamma_{L}}\left(\Omega\right)&\rightarrow&H^{1}_{\Gamma_{1
}\cup\Gamma_{L}}\left(\Omega\right)\nonumber \\
\eta&\mapsto&T\left(\eta\right)=\theta,
\end{eqnarray}
where $\theta$ is the unique solution of the linear problem \ref{pro3.2}.

We establish that $T$ is completely continuous. For given $\eta_1$ (resp.
$\eta_2$)
in $H^{1}_{\Gamma_{1}\cup\Gamma_{L}}\left(\Omega\right)$, we associate
$T(\eta_1)$ (resp. $T(\eta_2)$), the solution of the equation (\ref{37}). By
substraction we obtain,
\begin{eqnarray}\label{39}
B\left(T(\eta_1)-T(\eta_2), \psi\right)=L(\eta_1, \psi)-L(\eta_2, \psi),\quad
\forall \psi\in H^{1}_{\Gamma_{1}\cup\Gamma_{L}}\left(\Omega\right).
\end{eqnarray}
By taking $\psi=T(\eta_1)-T(\eta_2)$ in (\ref{39}) we get
\begin{eqnarray}\label{40}
\int_\Omega K(x)\left|\nabla \left(T(\eta_1)-T(\eta_2)\right)\right|^2dx+P=Q+R,
\end{eqnarray}
where
\begin{eqnarray*}
P&=&\int_\Omega\left(T(\eta_1)-T(\eta_2)\right)v_i\partial_i\left(T(\eta_1)-T(
\eta_2)\right)dx,\\Q&=&2\int_\Omega\left[\mu\left(\eta_1, v,
\left|D(v)\right|\right)-\mu\left(\eta_2, v,
\left|D(v)\right|\right)\right]\left[T(\eta_1)-T(\eta_2)\right]
\left|D(v)\right|^2dx,
\end{eqnarray*}
and
\begin{eqnarray*}
R=\int_\Omega
\left[r(\eta_1)-r(\eta_2)\right]\left[T(\eta_1)-T(\eta_2)\right]dx.
\end{eqnarray*}
Let us evaluate the terms $P$, $Q$ and $R$. By using the same arguments as in
(\ref{38}) we get
\begin{eqnarray}\label{41}
P=\frac{1}{2}\int_\Omega v_i \partial_i
\left[\left(T(\eta_1)-T(\eta_2)\right)^2\right]dx=\frac{1}{2}\int_\Omega
v_i\partial_i\left[\left(\theta_1-\theta_2\right)^2\right]dx=0.
\end{eqnarray}
For the trem $Q$ we write,
\begin{eqnarray}\label{42}
\left|Q\right|\leq2C_\mu
\int_\Omega\left|\eta_1-\eta_2\right|\left|T(\eta_1)-T(\eta_2)\right|\left|D(v)
\right|^2dx,
\end{eqnarray}
where $C_\mu$ is the Lipschitz constant of the function $t\mapsto\mu (t, .,
.)$.\newline\\ Note here that if we took $\theta$ and $\eta$ in
$W^{1,q}_{\Gamma_{1}\cup\Gamma_{L}}\left(\Omega\right)$, the integrand in
(\ref{42}) would not be necessairily in $L^1(\Omega)$, this forced us to take
$H^{1}_{\Gamma_{1}\cup\Gamma_{L}}\left(\Omega\right)$ instead of
$W^{1,q}_{\Gamma_{1}\cup\Gamma_{L}}\left(\Omega\right)$.\newline\\ Now we know
from Rellich-Kondrachov Theorem that
$H^{1}_{\Gamma_{1}\cup\Gamma_{L}}\left(\Omega\right)$ is compactly embedded in
$L^4\left(\Omega\right)$. Then by taking $\eta_i$ and $T(\eta_i)$ in
$L^4\left(\Omega\right)$ we get
\begin{eqnarray*}
\left|Q\right|\leq2 \left|\Omega\right|^{\frac{p-4}{2p}}C_\mu
\left\|\eta_1-\eta_2\right\|_4\left\|T(\eta_1)-T(\eta_2)\right\|_4\left\|D(v)
\right\|^{2}_{p}.
\end{eqnarray*}
Again by the compact embedding
$H^{1}_{\Gamma_{1}\cup\Gamma_{L}}\left(\Omega\right)\subset
L^4\left(\Omega\right)$, there exists a positive constant $C'$ depending only on
$\Omega$ such that
\begin{eqnarray}\label{43}
\left|Q\right|\leq2 \left|\Omega\right|^{\frac{p-4}{2p}}C_\mu
C'\left\|\eta_1-\eta_2\right\|_{1,2}\left\|T(\eta_1)-T(\eta_2)\right\|_{1,2}
\left\|D(v)\right\|^2_p.
\end{eqnarray}
Finally by H$\ddot {\mbox o}$lder's and Poincar\'e's inequalities
 we have
\begin{eqnarray}\label{44}
\left|R\right| \leq
C_r\left\|\eta_1-\eta_2\right\|_2\left\|T(\eta_1)-T(\eta_2)\right\|_2
\leq C_P
C_r\left\|\eta_1-\eta_2\right\|_{1,2}\left\|T(\eta_1)-T(\eta_2)\right\|_{1,2},
\end{eqnarray}
where $C_P$ and $C_r$ are respectively Poincar\'e's constant and Lipschitz's
constant of the function $r$. Now by (\ref{36}), (\ref{40}), (\ref{41}),
(\ref{43}) and (\ref{44}) we obtain
\begin{eqnarray*}
k_0\left\|T(\eta_1)-T(\eta_2)\right\|^2_{1,2}\leq\left(2\left|\Omega\right|^{
\frac{p-4}{2p}}C_\mu C'\left\|D(v)\right\|^2_p+C_P
C_r\right)\left\|\eta_1-\eta_2\right\|_{1,2}\left\|T(\eta_1)-T(\eta_2)\right\|_{
1,2},
\end{eqnarray*}
and then
\begin{eqnarray*}
k_0\left\|T(\eta_1)-T(\eta_2)\right\|_{1,2}\leq\left(2\left|\Omega\right|^{\frac
{p-4}{2p}}C_\mu C'\left\|D(v)\right\|^2_p+C_P
C_r\right)\left\|\eta_1-\eta_2\right\|_{1,2}.
\end{eqnarray*}
This proves that the operator $T$ is Lipschitzian. Let us now show that $T$ is
bounded in $H^{1}_{\Gamma_{1}\cup\Gamma_{L}}\left(\Omega\right)$.
 We know that for $\eta \in
H^{1}_{\Gamma_{1}\cup\Gamma_{L}}\left(\Omega\right)$, $T(\eta)$ is solution of
(\ref{37}), that is
\begin{eqnarray*}
B(T(\eta),\psi)=L(\eta, \psi),\quad \forall\psi\in
H^{1}_{\Gamma_{1}\cup\Gamma_{L}}\left(\Omega\right).
\end{eqnarray*}
By taking $\psi=T(\eta)$, and using (\ref{24}) and (\ref{36}) we obtain
\begin{eqnarray*}
k_0\int_\Omega\left|\nabla
T(\eta)\right|^2dx\leq2\mu_1\int_\Omega\left|D(v)\right|^2\left|T(\eta)\right|dx
+r_1\int_\Omega\left|T(\eta)\right|dx+\int_\omega\left|\theta_\omega
\right|\left|T(\eta)\right|dx,
\end{eqnarray*}
where $r_1=$ ess $\sup \left\{r(t), t\in\R\right\}$.
From the continuous embedding
$H^1_{\Gamma_{1}\cup\Gamma_{L}}\left(\Omega\right)\subset L^2(\omega)$, there
exists a positive constant $C''$ independent of $\eta$ such that
\begin{eqnarray*}
\left\|T(\eta)\right\|_{L^2(\omega)}\leq C''\left\|T(\eta)\right\|_{1,2},
\end{eqnarray*}
and by H$\ddot{\mbox{o}}$lder's and Poincar\'e's inequalities, we get
\begin{eqnarray}\label{45}
k_0
\left\|T(\eta)\right\|^2_{1,2}&\leq&2\left|\Omega\right|^{\frac{p-4}{2p}}\mu_1
C_p \left\|D(v)\right\|^2_p\left\|T(\eta)\right\|_{1,2}\nonumber\\&+&r_1
C_p\left\|T(\eta)\right\|_{1,2}+C''\left\|\theta_\omega
\right\|_{L^2(\omega)}\left\|T(\eta)\right\|_{1,2}.
\end{eqnarray}
The boundedness of $T$ follows from (\ref{45}),
$\left\|T(\eta)\right\|_{1,2}\leq C^*$,
where
\begin{eqnarray*}
C^*=k^{-1}_0\left(2\left|\Omega\right|^{\frac{p-4}{2p}}\mu_1 C_p
\left\|D(v)\right\|^2_p+C''\left\|\theta_\omega \right\|_{L^2(\omega)}+C_p
r_1\right).
\end{eqnarray*}
Now according to Schauder Fixed Point Theorem we deduce that the operator $T$
has at least
one fixed point $\theta\in H^1_{\Gamma_{1}\cup\Gamma_{L}}\left(\Omega\right)$,
solution of the variational problem (\ref{Sprob}).
\end{proof}

In the following theorem, we prove the uniqueness of the solution to the
intermediate
{\rm Problem \ref{Sprob}}.

\begin{theorem}\label{thm3.5}
Let $\theta_\omega \in L^2\left(\omega\right)$ and $r \in
L^\infty\left(\R\right)$.
Assume that functions $r$ and $t\mapsto\mu (t, ., .)$ are Lipschitzian and
nonincreasing.
So with the same assumptions as in {\rm Theorem \ref{thm3.3}},
 the solution of  the intermediate
{\rm Problem \ref{Sprob}} is unique.
\end{theorem}

\begin{proof}
Indeed, suppose in contrary that there exist two solutions $\theta_1$ and
$\theta_2$ for (\ref{Sprob}).
By substracting, we obtain for all $\psi\in
H^1_{\Gamma_{1}\cup\Gamma_{L}}\left(\Omega\right)$,
\begin{eqnarray}\label{46}
\int_{\Omega}K(x) \nabla\Theta.\nabla\psi dx+\int_{\Omega} \psi
v_{i}\partial_{i}\Theta
dx&=&2\int_{\Omega}\left[\mu\left(\theta_1,v,
\left|D\left(v\right)\right|\right)-\mu\left(\theta_2,v,
\left|D\left(v\right)\right|\right)\right]\left|D\left(v\right)\right|^2 \psi
dx\nonumber\\&+&\int_{\Omega}
\left(r\left(\theta_1\right)-r\left(\theta_2\right)\right)\psi dx,
\end{eqnarray}
where $\Theta=\theta_1-\theta_2$. Now we use the real function $f_{\delta}$ (see
eg \cite{br, chp}) defined for $\delta>0$ by
\begin{eqnarray*}
f_{\delta}(t)=
\left\{
\begin{array}{ll}
(1-\frac{\delta}{t})^{+} \quad if \quad t>0,\\
 \quad 0 \quad \quad \quad if \quad t\leq0,
\end{array}
\right.
\end{eqnarray*}
with $A^+=\max(A,0)$. As $\Theta \in
H^1_{\Gamma_{1}\cup\Gamma_{L}}\left(\Omega\right)$ then $f_\delta(\Theta)\in
H^1_{\Gamma_{1}\cup\Gamma_{L}}\left(\Omega\right)$ and
\begin{eqnarray*}
\nabla f_\delta(\Theta)=\frac{\delta}{\Theta^2}\chi_{[\Theta>\delta]}\nabla
\Theta,
\end{eqnarray*}
here $\chi_{[\Theta>\delta]}$ is the indicator function of the set
$[\Theta>\delta]=\left\{x\in\Omega,\quad \Theta(x)>\delta\right\}$, that is
\begin{eqnarray*}
\chi_{[\Theta>\delta]}(x)=1\quad if\quad \Theta(x)>\delta,\quad
\chi_{[\Theta>\delta]}(x)=0 \quad if \quad \Theta(x)\leq\delta.
\end{eqnarray*}
As $\psi=0$ on $\Gamma_{1}\cup\Gamma_{L}$, $v.n=0$ on $\omega$ and $div(v)=0$ in
$\Omega$, we have
 \begin{eqnarray}\label{h}
\int_{\Omega} \psi v_{i}\partial_{i}\Theta dx=-\int_{\Omega} \Theta v.\nabla\psi
dx.
\end{eqnarray}
Taking $\psi=f_\delta(\Theta)$ in (\ref{46}) and using (\ref{h}) we get
\begin{eqnarray}\label{47}
\delta\int_{\Omega\cap \left[\Theta>\delta\right]}K(x)
\left|\frac{\nabla\Theta}{\Theta}\right|^2dx&=&2\int_{\Omega\cap
\left[\Theta>\delta\right]}\frac{\mu\left(\theta_1,v,
\left|D\left(v\right)\right|\right)-\mu\left(\theta_2,v,
\left|D\left(v\right)\right|\right)}{\theta_1-\theta_2}
\left|D\left(v\right)\right|^2 \left(\Theta-\delta\right)
dx\nonumber\\&+&\int_{\Omega\cap
\left[\Theta>\delta\right]}\frac{r\left(\theta_1\right)-r\left(\theta_2\right)}{
\theta_1-\theta_2}\left(\Theta-\delta\right) dx+\delta\int_{\Omega\cap
\left[\Theta>\delta\right]} v \frac{\nabla\Theta}{\Theta} dx.
\end{eqnarray}
Since the functions $r$ and $t\mapsto\mu (t, ., .)$ are nonincreasing then
\begin{eqnarray}\label{Lip}
\frac{\mu\left(\theta_1,v,\left|D\left(v\right)\right|\right)-\mu\left(\theta_2,
v,\left|D\left(v\right)\right|\right)}{\theta_1-\theta_2}\leq0, \quad \mbox{and}
\quad
\frac{r\left(\theta_1\right)-r\left(\theta_2\right)}{\theta_1-\theta_2}\leq0.
\end{eqnarray}
 Recall that $v\in V^{p}_{div} \subset
\left(L^{\infty}\left(\Omega\right)\right)^3$, so there exists a positive
constant $M$ independent of $\delta$ such that
\begin{eqnarray*}
\int_{\Omega\cap \left[\Theta>\delta\right]} v \frac{\nabla\Theta}{\Theta}
dx\leq M\int_{\Omega\cap \left[\Theta>\delta\right]}
\left|\frac{\nabla\Theta}{\Theta}\right|dx.
\end{eqnarray*}
Now by (\ref{36}), (\ref{Lip}) and Cauchy-Schwarz inequality, (\ref{47}) becomes
\begin{eqnarray*}
k_0\int_{\Omega\cap \left[\Theta>\delta\right]}
\left|\frac{\nabla\Theta}{\Theta}\right|^2dx\leq M\int_{\Omega\cap
\left[\Theta>\delta\right]} \left|\frac{\nabla\Theta}{\Theta}\right|dx\leq
M\left|\Omega\right|^{\frac{1}{2}}\left(\int_{\Omega\cap
\left[\Theta>\delta\right]}\left|\frac{\nabla\Theta}{\Theta}\right|^2dx\right)^{
\frac{1}{2}}.
\end{eqnarray*}
Then
\begin{eqnarray}\label{48}
\left(\int_\Omega\left|\nabla
\ln\left(1+\frac{(\Theta-\delta)^+}{\delta}\right)\right|^2dx\right)^\frac{1}{2}
&=&\left(\int_{\Omega\cap
[\Theta>\delta]}\left|\frac{\nabla\Theta}{\Theta}\right|^2dx\right)^{\frac{1}{2}
}\nonumber\\&\leq& Mk^{-1}_0 \left|\Omega\right|^{\frac{1}{2}}.
\end{eqnarray}
The right hand side of (\ref{48}) is independent of $\delta$,
then for $\delta\rightarrow0$ we must obtain, $\Theta=\theta_1-\theta_2\leq0$
a.e. in $\Omega$, and by permuting the roles of $\theta_1$ and $\theta_2$
 we get $\theta_2-\theta_1\leq0$ then $\theta_1=\theta_2$.
This ends the proof of uniqueness of the temperature.
\end{proof}

\subsection{Existence result for the coupled problem 2.1}\label{sec3.3}
We recall here the necessary assumptions to ensure the existence of at least one
solution to the coupled problem \ref{prob2.1}.

We assume that, the real number $p\geq 4$, the function $\mu$ satisfies
(\ref{23}) and (\ref{24}),
 the function $K$ satisfies (\ref{36}),
the exterior force vector $f\in(W^{1,p}(\Omega))^3$,
the upper limit for the stress $0\leq k\in L^p(\omega)$,
the given fixed flux $\theta_\omega$ on $\omega$ is in $L^{2}(\omega)$,
the real function $r \in L^{\infty}(\R)$, we also suppose that functions
$r$ and $t\mapsto\mu(t, ., .)$ are Lipschitzian and nonincreasing.

\begin{theorem}\label{thm3.6}
Under the above assumptions, there exists a unique
$\theta \in H^{1}_{\Gamma_{1}\cup\Gamma_{L}}\left(\Omega\right)$ solution to
{\rm Problem \ref{Sprob}} and there exists at least one
$\left(v_{\theta}, \pi_{\theta}\right)\in V^p_{div}\times
L^p_0\left(\Omega\right)$
satisfying the variational inequality {\rm(\ref{14})}.
\end{theorem}
\begin{proof}
For all $\eta \in H^{1}_{\Gamma_{1}\cup\Gamma_{L}}\left(\Omega\right)\subset
W^{1 , q}_{\Gamma_{1}\cup\Gamma_{L}}(\Omega)$,
because $1<q<2$,
by {\rm Theorem \ref{th3.1}},
 there exists $v=v_{\eta}$ in $V^{p}_{div}$ and by {\rm Theorem \ref{th3.2}}
there exists  $\pi=\pi_{\eta}$ in $L^{p}_{0}(\Omega)$
solution to the  variational inequality
\begin{eqnarray}
 a\left(\eta, v_\eta, \varphi-v_{\eta}\right)-\left(\pi_{\eta},
div\left(\varphi\right) \right) + j\left(\varphi \right)-j\left(v_{\eta} \right)
\geq \left(f, \varphi-v_{\eta}\right)
 \qquad \forall \varphi\in V^{q},
\end{eqnarray}
also by {\rm Theorems \ref{thm3.4}-\ref{thm3.5}}, there exists a unique  $\theta
\in H^{1}_{\Gamma_{1}\cup\Gamma_{L}}\left(\Omega\right)$,
 solution to the  {\rm Problem \ref{Sprob}}. So we can use the oprerator $T$
defined by (\ref{opT}). By {\rm Theorem \ref{thm3.4}}
we know that $T$ has at least one fixed point $\theta \in
H^{1}_{\Gamma_{1}\cup\Gamma_{L}}\left(\Omega\right)$, $\theta=T(\theta)$, which
is solution to the {\rm Problem \ref{Sprob}}. Then
$\left(\theta, v_{\theta}, \pi_{\theta}\right)$ is solution to the {\rm Problem
\ref{prob2.1}}.
\end{proof}

 \begin{rem}
To our knowledge, the uniqueness of the problem {\rm\ref{prob2.1}} remains an open question.
 \end{rem}

%%%% Acknowledgments %%%%%%%%
\section*{Acknowledgments}
The authors would like to thank the anonymous referee.

%%-----------------------------
%%      your bibliography
%%-----------------------------


\begin{thebibliography}{00}
\bibitem{Amrgir} { C. Amrouche, V.Girault, }
{ Decomposition of vector spaces and application to the Stokes problem in
arbitrary dimension. }
Czechoslovak Mathematical Journal, 44 (119) (1994), Praha 119-140.

\bibitem{gb-mb1}{ G.Bayada, M. Boukrouche, }
{ On a free boundary problem for Reynolds equation derived from
 the Stokes system with Tresca boundary conditions. }
Journal of Mathematical Analysis and Applications, Vol. 282(1), (2003) 212-231.

\bibitem{MB-b} { F. Boughanim, M. Boukrouche, H. Smaoui, }
{ Asymptotic behavior of a non-Newtonian flow with stick-slip
condition.} Electronic Journal of Differential Equations, Conference 11 (2004)
71-80.

\bibitem{tapiero2} { F. Boughanim, R. Tapi\'ero, }
{ Derivation of the two-dimensional Carreau law for a quasi-Newtonian fluid flow
through a thin slab. }
Applicable Analysis: An International Journal, Vol. 57(3), (1995) 243-269.

\bibitem{MB2009}  { M. Boukrouche, }
{ A brief survey on lubrication problems with nonlinear boundary conditions. }
MAT. Serie A: Mathematical Conferences, Seminars and Papers, ( Universidad
Austral, Facultad de Ciencias Empresariales Departamento de Matem\'atica,
Rosario), 16 (2009).

\bibitem{MB-RM3} { M. Boukrouche, R. El Mir, }
{ Asymptotic analysis of a non-Newtonian fluid in a thin domain with Tresca
law.}
Nonlinear Analysis, Theory Methods and Applications, Vol. 59(1-2), (2004)
85-105.

 \bibitem{MB-RM2} { M. Boukrouche, R. El Mir, }
{ On the Navier-Stokes system in a thin film flow
 with Tresca free boundary condition and its asymptotic behavior.}
 Bull. Math. Soc. Sc. Math. Roumanie, Tome 48 (96) (2), (2005) 139-163.

\bibitem{MB-RM1} { M. Boukrouche, R. El Mir, }
{ Non-isothermal, non-Newtonian lubrication problem
 with Tresca fluid-solid law. Existence and
 asymptotic of weak solutions.}
Nonlinear Analysis, Real World Applications, Vol. 9 (2), (2008) 674-692.

\bibitem{sf1} { M. Boukrouche, F. Saidi, }
{ Non-isothermal lubrication problem with
Tresca fluid-solid interface law. Part I.}
Nonlinear Analysis, Real World Applications, Vol.7(5),(2006) 1145-1166.

\bibitem{br} {\sc H. Br\'ezis, D. Kinderlehrer, and  G. Stampacchia, }
{\em Sur une nouvelle formulation du probl\`eme de l'\'ecoulement \`a travers
une digue.}
C.R.A.S. Paris S\'erie A (287), (1978) 711-714.

\bibitem{malek2} M. Bul\'icek, M. Majdoub, and J. M\'alek,
{Unsteady flows of fluids with pressure dependent viscosity in
unbounded domains.} Nonlinear Analysis: Real World Applications 11 (2010)
3968-3983

\bibitem{malek1} M. Bul\'icek, J. M\'alek,  and K. R. Rajagopal,
{ Analysis of the flows of incompressible fluids with pressure dependent
viscosity fulffilling $\nu(p , \cdot) \to +\infty$ as $p\to +\infty$.}
Czechoslovak Mathematical Journal, 59 (134) (2009), 503-528.

\bibitem{gwia} { M. Bul\'icek, J. M\'alek,  and A. \'Swierczewska-Gwiazda, }
{  On steady flows of incompressible fluids with implicit power-law-like
rheology.}
Advances in calculus of variations 2 (2009), no. 2, 109-136.

\bibitem{chp}{ M. Chipot, G. Michaille, }
{ Uniqueness results and monotonicity properties for the solution of some
variational inequalities.}
Annali della Scuola Norm. Sup. Pisa, (1989) 137-166.

\bibitem{dlions} { G. Duvaut,  J.L. Lions, }
{ Les In\'equations en m\'ecanique et en physique.}
Dunod, Paris, (1972).

\bibitem{ge}  { M. Fang, R.P. Gilbert, }
{ Nonlinear systems arising from non-isothermal, non-Newtonian Hele-Shaw flows
in the presence of body forces and sources. }
Mathematical and Computer Modelling, 35 (2002) 1425-1444.

\bibitem{gr} { V. Girault,  P.A. Raviart, }
{ Finite element Approximation of the Navier-Stokes Equations. }
Springer-Verlag, (1979).


\bibitem{J.L} { J.L. Lions, }
{ Sur quelques m\'ethodes de r\'esolution des probl\`emes aux limites non
lin\'eaires. }Dunod Gauthier-Villars, Paris (1969).

\bibitem{majda1984} { A. Majda, }
{ Compressible Fluid flow and Systems of conservation laws in several  space
variables. }
Applied Mathematical Sciences 53, Springer-Verlag, 1984.

\bibitem{tapiero} {A. Mikeli\'c,  R. Tapi\'ero, }
{ Mathematical derivation of the power law describing polymer flow through a
thin slab. }RAIRO Mod. Math. Anal. 29(1) (1995) 3-21.
\end{thebibliography}
\end{document}